\documentclass{amsart}
\usepackage{amssymb,amsmath, amsthm,latexsym}
\usepackage{psfrag}
\usepackage{amscd}
\usepackage{graphicx}
\newcommand{\cal}[1]{\mathcal{#1}}
\theoremstyle{plain}
\newtheorem*{theo}{Theorem}

\newtheorem{lemma}{Lemma}[section]

\newtheorem{proposition}[lemma]{Proposition}

\theoremstyle{definition}

\let\egthree=\phi
\let\phi=\varphi
\let\varphi=\egthree

\begin{document}
\title{Distance in the curve graph}
\author{Ursula Hamenst\"adt}
\thanks{Partially supported by the Hausdorff Center Bonn}
\date{May 17, 2012}

\begin{abstract}
We estimate the distance
in the curve graph of a surface $S$ of finite type 
up to a fixed
multiplicative constant using Teichm\"uller geodesics.
\end{abstract}

\maketitle

\section{Introduction}

The \emph{curve graph} of an oriented surface $S$ of genus $g\geq 0$
with $m\geq 0$ punctures and $3g-3+m\geq 2$ 
is the metric graph $({\cal C\cal G},d_{\cal C\cal G})$
whose vertices 
are isotopy classes of 
essential (i.e. non-contractible and not 
puncture parallel) simple
closed curves on $S$. Two such 
curves are connected by an edge of length one if and only
if they can be realized disjointly. 
The curve graph 
is a hyperbolic geodesic metric space of 
infinite diameter \cite{MM99}.
It turned out to be an important tool
for understanding the geometry of the 
\emph{mapping class group} of $S$
\cite{MM00}.

The \emph{Teichm\"uller space} 
${\cal T}(S)$ of $S$ is the space 
of all complete
finite volume marked hyperbolic metrics on $S$.
Let $P:{\cal Q}^1(S)\to {\cal T}(S)$ be the 
bundle of marked area one \emph{quadratic differentials}. 
The bundle ${\cal Q}^1(S)$ can naturally be 
identified with the unit cotangent bundle of ${\cal T}(S)$
for the \emph{Teichm\"uller metric}. 
In particular, every $q\in {\cal Q}^1(S)$ determines
the unit cotangent line of a \emph{Teichm\"uller geodesic}. 

An area one quadratic differential
$q$ defines a \emph{singular euclidean metric}
on $S$ of area one.
Call a simple closed curve $\alpha$
\emph{$\delta$-wide} for $q$
if $\alpha$ is the
core curve of an embedded annulus of width at least 
$\delta$ with respect to the metric $q$. Here the width
of an annulus is the minimal distance between 
the two boundary circles. 
For small enough $\delta$, a $\delta$-wide curve exists 
for all $q\in {\cal Q}^1(S)$
(Lemma 5.1 of \cite{MM99}).

A map
$\Upsilon:{\cal Q}^1(S)\to {\cal C\cal G}$
which associates to a marked area one
quadratic differential $q$ a 
$\delta$-wide curve for $q$
relates the Teichm\"uller metric on ${\cal T}(S)$
to the metric on ${\cal C\cal G}$.  Namely, 
there is a number $L>1$ 
such that for every 
unit cotangent line $(q_t)\subset {\cal Q}^1(S)$ 
of a Teichm\"uller geodesic, the assignment 
$t\to \Upsilon(q_t)$ is a coarsely Lipschitz
\emph{unparametrized $L$-quasi-geodesic} 
(this is an immediate consequence of 
Theorem 2.6 of \cite{MM99}).
This implies that there is a number $\kappa>0$ 
not depending on the geodesic such that
\begin{align}\label{kappa}
d_{\cal C\cal G}(\Upsilon(q_s),\Upsilon(q_t))& 
\leq \kappa\vert t-s\vert +\kappa
\text{ and }\\
d_{\cal C\cal G}(\Upsilon(q_s),\Upsilon(q_t))
+d_{\cal C\cal G}(\Upsilon(q_t),\Upsilon(q_u))
& \leq d_{\cal C\cal G}(\Upsilon(q_s),\Upsilon(q_u))+\kappa\notag
\end{align}
for $s\leq t\leq u$ (see Lemma 2.5 of \cite{H10} for 
a detailed discussion).

In general these unparametrized
quasi-geodesics are not quasi-geodesics with their
proper parametrizations. More precisely, 
for $\epsilon>0$ let 
${\cal T}(S)_\epsilon\subset {\cal T}(S)$ be the set  
of all complete finite volume marked hyperbolic 
metrics on $S$ whose \emph{systole}
(i.e. the shortest length of a simple closed
geodesic) is at least $\epsilon$. 
Then for a given number $L^\prime>L$, there exists a number
$\epsilon=\epsilon(L^\prime)>0$ such that the unparametrized 
quasi-geodesic in ${\cal C\cal G}$ 
defined by the image under $\Upsilon$ of the cotangent line
of a Teichm\"uller geodesic $\gamma$ 
is a \emph{parametrized} $L^\prime$-quasi-geodesic
only if $\gamma$ entirely remains in
${\cal T}(S)_\epsilon$
\cite{MM99,H10}. In particular, 
$d_{\cal C\cal G}(\Upsilon(q_s),\Upsilon(q_t))$ may be
uniformly bounded even though $\vert s-t\vert$ is arbitrarily large.

Nevertheless, the estimate (\ref{kappa}) allows to construct
for any two points $\xi,\zeta\in {\cal C\cal G}$ a parametrized
uniform quasi-geodesic connecting
$\xi$ to $\zeta$ as follows. 

Note first that it is 
very easy to decide whether the distance in the curve
graph between $\xi$ and $\zeta$ is at 
most two. Namely, in this case there is an essential simple
closed curve $\alpha$ which can be
realized disjointly from $\xi,\zeta$.

If the distance between $\xi,\zeta$ is at least three then 
$\xi,\zeta$ 
define the unit cotangent line $(q_t)$ of a Teichm\"uller geodesic  
whose \emph{horizontal} and \emph{vertical}  measured foliations
are one-cylinder foliations with core curves homotopic to 
$\xi,\zeta$, respectively.
For any sequence $t_0<\dots <t_n$ so that
$d_{\cal C\cal G}(\Upsilon(q_{t_i}),\Upsilon(q_{t_{i+1}}))\geq 2\kappa$
for all $i$ and $d_{\cal C\cal G}(\Upsilon(q_s)),\Upsilon(q_t))<2\kappa$
for all $s,t\in (t_i,t_{i+1})$,
a successive application of the second part of the 
estimate (\ref{kappa}) implies that
\[d_{\cal C\cal G}(\Upsilon(q_{t_0}),\Upsilon(q_{t_n}))\geq n\kappa.\]
Combined  with the first part of (\ref{kappa}), 
the map $s\in [i,i+1)\to \Upsilon(q_{t_i})$ is a uniform
quasi-geodesic in ${\cal C\cal G}$.

Unfortunately, up to date no bounds for 
the number $\kappa>0$ are known.
Moreover, the curve graph is locally infinite, and 
although there are algorithms which compute
the distance between two curves in the
curve graph explicitly (see \cite{S04} for an example of such an
algorithm), these algorithms depend on the calculation
of so-called \emph{tight geodesics}. 
 Unlike the problem of deciding
whether or not the distance in the curve graph is
at least three, 
determining whether or not
the distance in ${\cal C\cal G}$ between two given curves is
bigger than a given number $\beta\geq 4$
does not seem to be easier than calculating their
distance explicitly.

The goal of this note is to show that 
we can use Teichm\"uller geodesics to calculate the distance
between two points in ${\cal C\cal G}$ up to a fixed multiplicative
constant by only detecting pairs of curves
of distance at least three. We show

\begin{theo}\label{main}
There is a number $\theta>0$ with the following property.
Let $(q_t)\subset {\cal Q}^1(S)$ be the unit cotangent line of 
a Teichm\"uller geodesic. 
Assume that
there are numbers $0=t_0<\dots <t_n=T$ such that for all 
$i<n$ we have 
\begin{enumerate}
\item $t_{i+1}\geq t_{i}+1$ 
\item $d_{\cal C\cal G}(\Upsilon(q_{s}),\Upsilon(q_{t}))\geq 3$ 
for $s\leq t_i, t\geq t_{i+1}$.
\end{enumerate}
Then $d_{\cal C\cal G}(\Upsilon(q_0),\Upsilon(q_T))\geq  n/\theta-\theta$.
\end{theo}

Let again $(q_t)$ be the unit cotangent line 
of a Teichm\"uller geodesic. For $T>0$,
construct recursively a sequence
$0=t_0<\dots <t_u=T$ as follows. If $t_i$ has already
been determined, let $t_{i+1}$ be the smallest number contained in the
interval 
$[t_i+1,T]$ so that 
$d_{\cal C\cal G}(\Upsilon(q_{s}),\Upsilon(q_{t}))\geq 3$
for all $s\leq t_i,t\geq t_{i+1}$. If no such number exists put $u=i$.
By the first part of the 
estimate (\ref{kappa}), 
for each $s\leq t_\ell$ there is at least one $t\leq t_{\ell-1}$ so that
$d_{\cal C\cal G}(\Upsilon(q_{s}),\Upsilon(q_{t}))\leq 3+2\kappa$.
A successive application of the triangle inequality shows that
$d_{\cal C\cal G}(\Upsilon(q_{T}),\Upsilon(q_0))\leq u(3+2\kappa)+3$.
On the other hand,  we have 
 $d_{\cal C\cal G}(\Upsilon(q_{T}),\Upsilon(q_0))\geq u\theta-1/\theta$ by 
the Theorem.

The proof of this result 
relies on the work of Minsky \cite{M92} and
Rafi \cite{R10}. The basic idea is that $\delta$-wide
curves for a quadratic differential $q$ detect "active" 
subsurfaces in the
sense of \cite{R10} for the Teichm\"uller geodesic
 defined by $q$. However, such active subsurfaces
may intersect in a complicated way, and making this
simple idea rigorous and quantitative is the main 
contribution of this work. Unfortunately
the proof of the theorem does not give an
explicit bound for the number $\theta$.



\section{Teichm\"uller geodesics and  the curve graph}

Let $S$ be an oriented surface of
genus $g\geq 0$ with $m\geq 0$ punctures and 
of \emph{complexity} $\xi(S)=3g-3+m\geq 1$.
In the sequel we always mean by a simple closed curve 
on $S$ an essential simple closed curve.
Moreover, we identify curves which are isotopic unless
specifically stated otherwise. 

We denote by ${\cal C\cal G}$ the curve graph of $S$ and by  
$d_{\cal C\cal G}$ the distance in ${\cal C\cal G}$.
If $\xi(S)=1$ then $S$ is a once-punctured torus or a
four-punctured sphere and ${\cal C\cal G}$ is the
Farey graph (i.e. two simple closed curves are
connected by an edge if 
they intersect in the 
smallest possible number of points).

Let $P:{\cal Q}^1(S)\to {\cal T}(S)$ be the 
bundle of marked area one  
quadratic differentials over the
Teichm\"uller space ${\cal T}(S)$ of $S$.
Such a quadratic differential on a surface $x\in {\cal T}(S)$ 
is a meromorphic section of
$T^\prime(x)\otimes T^\prime(x)$ with at most simple poles
at the punctures and no other poles. It
defines a singular euclidean metric on $S$ of finite
area, and we require that the area of this metric equals one
(see \cite{S84} for details).
Any two simple closed curves
$\xi,\zeta$ of distance at least $3$ in ${\cal C\cal G}$ define
a one-parameter family $(q_t)$ of area one quadratic differentials 
which form the unit cotangent line of a 
Teichm\"uller geodesic.
These differentials are determined by the
requirement
that the \emph{vertical} and the \emph{horizontal}
measured foliation of $q_t$ is a one-cylinder 
foliation with cylinder curves homotopic to 
$\xi,\zeta$, respectively (see Proposition 3.10 of
\cite{H12} for one of the many places in the 
literature where such quadratic differentials are
constructed explicitly).

For $x\in {\cal T}(S)$ and 
an essential simple closed curve $\alpha$ on $S$ let 
${\rm Ext}_x(\alpha)$ be the \emph{extremal length} of $\alpha$
with respect to the conformal structure defined by $x$.
If $q$ is a quadratic differential  on  $x$ then 
we also write ${\rm Ext}_q(\alpha)$ instead of
${\rm Ext}_x(\alpha)$. We refer to \cite{G87} for more
and references. 

In the sequel we always mean by a quadratic differential a
marked area one quadratic differential. 
The $q$-length $\ell_q(\alpha)$ 
of an essential
simple closed curve $\alpha$ on $S$ is the infimum of the
lengths for the singular euclidean 
metric defined by $q$ of all curves freely
homotopic to $\alpha$. With this terminology,  we have
\[{\rm Ext}_q(\alpha)\geq \ell_q(\alpha)^2\]
for every simple closed curve $\alpha$.

There is a number $\delta>0$ such that
each quadratic differential $q$ admits a 
$\delta$-wide curve \cite{MM99} (see also
\cite{B06} for a more explicit and detailed
discussion), i.e. a simple
closed curve $\alpha$ which is the core curve
of an embedded annulus of width at least $\delta$.
Then the $q$-length of $\alpha$ is at most $1/\delta$.
Define a map 
\[\Upsilon:{\cal Q}^1(S)\to {\cal C\cal G}\]
by associating 
to an area one quadratic differential $q$ a 
$\delta$-wide simple closed 
curve for $q$. 

There is a number $\kappa>0$ such that 
for each cotangent line $(q_t)$ of a Teichm\"uller geodesic,
the estimate (\ref{kappa}) from the introduction holds true.
In particular, if $t>s,k\geq 0$ are such that
$d_{\cal C\cal G}(\Upsilon(q_s),\Upsilon(q_t))\geq k+2\kappa$
then \[d(\Upsilon(q_u),\Upsilon(q_v))\geq k\]
for all $u\leq s<t\leq v$.

For a quadratic differential $q\in {\cal Q}^1(S)$ let 
\[\nu(q)=\inf\{{\rm Ext}_q(\alpha)\mid \alpha\}\]
where $\alpha$ ranges through the essential simple closed
curves. Note that the function $\nu:{\cal Q}^1(S)\to (0,\infty)$
is continuous. Moreover, for every $\epsilon >0$, the
set $\{\nu\geq \epsilon\}$ is closed and invariant
under the action of the \emph{mapping class group}
${\rm Mod}(S)$ of $S$, and the quotient
$\{\nu\geq \epsilon\}/{\rm Mod}(S)$ is compact.

Denote by $\lambda$ the Lebesgue
measure on the real line. 
The following observation can be thought of 
as a quantitative account on progress in 
the curve graph along Teichm\"uller geodesic segments
which spend a definitive proportion of time in the thick
part of Teichm\"uller space.

\begin{proposition}\label{progression}
For every $\epsilon>0,k>0$ and every $b\in (0,1)$
there is a number
$R_0=R_0(\epsilon,k,b)>0$ with the following property.
Let $(q_t)$ be the unit cotangent line of 
a Teichm\"uller geodesic.
If $\lambda\{t\in [0,R]\mid \nu(q_t)\geq \epsilon\}\geq bR$
for some $R\geq R_0$ 
then $d_{\cal C\cal G}(\Upsilon(q_0),\Upsilon(q_{R}))\geq k.$
\end{proposition}
\begin{proof} Let $\epsilon>0,k>0,b\in (0,1)$. 
We show first that there
is a number 
\[\tau=\tau(\epsilon,k,b)>1\] with the following 
property. For every Teichm\"uller geodesic $(q_t)$ there is
a number $\sigma\in[1,\tau]$
so that either 
$d_{\cal C\cal G}(\Upsilon(q_0),\Upsilon(q_\sigma))\geq k$ or
\[\lambda\{s\in [0,\sigma]\mid \nu(q_s)< \epsilon\}\geq \sigma(1-b).\]

For this we argue by contradiction and we assume that the
claim does not hold. Then there is a sequence $(q^i_t)$ 
of unit cotangent lines of Teichm\"uller geodesics 
so that for each $i>0$, 
the smallest subinterval of $[0,\infty)$ containing $[0,1]$ 
with the property that the above alternative holds true
for $(q^i_t)$  
is of length at least $i$.

Since for each $i\geq 2$ the above alternative does not hold for 
$(q^i_t)$ and $\sigma=1$, there is at least one
$s\in [0,1]$ so that $\nu(q^i_s)\geq \epsilon$. This implies that
the quadratic differentials $q_0^i$ project to a compact subset
of the moduli space of quadratic differentials. 
Thus by invariance under the mapping class group,
after passing to a subsequence we may assume
that the quadratic differentials $q^i_0$ converge as $i\to \infty$
in ${\cal Q}^1(S)$ to a 
quadratic differential $q$. 

Let $\chi>0$ be sufficiently large that for every
quadratic differential $z$, the distance in the curve graph
between any two curves on $S$ of $z$-length at most $2/\delta$ 
is smaller than $\chi$ (see Lemma 2.1 of \cite{H10} 
for a detailed proof of the existence
of such a number).
Let $(q_t)$ be the unit cotangent line
of the Teichm\"uller geodesic with initial velocity $q_0=q$.

If $(q_t)$ is 
\emph{recurrent} (i.e. its projection to the
moduli space of quadratic differentials returns
to a fixed compact set for arbitrarily large times) 
then there is some $T>0$ 
so that
\[d_{\cal C\cal G}(\Upsilon(q_0),\Upsilon(q_T))\geq k+2\chi\]
(see Proposition 2.4 of \cite{H10} for details).
However, by continuity, 
for all sufficiently large $i$ the $q$-length (or $q_T$-length) 
of a curve of 
$q^i_0$-length at most $1/\delta$ (or of a curve of $q^i_T$-length 
at most $1/\delta$) does not exceed $2/\delta$. 
Thus by the choice of $\chi$, 
\[d_{\cal C\cal G}(\Upsilon(q^i_0), \Upsilon(q^i_T))\geq k\]
for all sufficiently large $i$ 
which violates the assumption on the sequence $(q^i_t)$.

On the other hand, if $(q_t)$ is not recurrent then 
there is some $t_0>0$ so that $\nu(q_t)\leq \epsilon/2$ 
for all $t\geq t_0$. Then the proportion of time the 
arc $q_t$ $(t\in[0,t_0/b])$ spends in the region $\{\nu\leq \epsilon/2\}$
is at least $1-b$. By continuity,
for all sufficiently large $i$ the proportion
of time the arc $q_t^i$ $(t\in[0,t_0/b])$ spends in the 
region $\{\nu<\epsilon\}$ is at least $1-b$ as well.
Once again, this is a contradiction.

Let $\kappa>0$ be as in (\ref{kappa}).
Let $b\in (0,1), k>0,\epsilon>0$, 
let $\tau=\tau(\epsilon,b/2,k+2\kappa)$ 
and let $R> 2\tau/b$. Let $(q_t)$ 
be the unit cotangent line of any Teichm\"uller geodesic.
There are successive numbers $0=t_0<t_1<\dots<t_u\leq R$ so that
$R-t_u\leq \tau$ and that 
for all $i<u$, $1\leq t_{i+1}-t_i\leq \tau$ and either 
\begin{equation}\label{innerinterval}
d_{\cal C\cal G}(\Upsilon(q_{t_i}),\Upsilon(q_{t_{i+1}}))\geq 
k+2\kappa\end{equation}
or $\lambda\{t\in [t_i,t_{i+1}]\mid \nu(q_t)<\epsilon\}\geq 
(t_{i+1}-t_i)(1-b/2).$

By the choice of $\kappa$, if there is some $i<u$ such that
the inequality (\ref{innerinterval}) holds true then
$d_{\cal C\cal G}(\Upsilon(q_0),\Upsilon(q_R))\geq k$.
Otherwise since  $R-t_u\leq \tau< bR/2$, we have
$\lambda\{t\in [0,R]\mid \nu(t)<\epsilon\}> R(1-b)$.
The proposition follows.
\end{proof}

For small $\epsilon >0$, the (hyperbolic) length of a simple
closed curve on a hyperbolic surface $x\in {\cal T}(S)$
is roughly proportional to its extremal length. Thus by the 
collar lemma, 
there is a number $\epsilon_0<\delta^2/2$ such
that for $x\in {\cal T}(S)$, any two simple closed
curves on $(S,x)$ of extremal length
at most $\epsilon_0$ can be realized
disjointly.

For $\epsilon \leq \epsilon_0$ let 
\[{\cal A}_\epsilon={\cal A}_\epsilon(x)\] 
be the set of all simple closed 
curves on $(S,x)$ of extremal length less than $\epsilon$. Then 
$S-{\cal A}_\epsilon$ is a union ${\cal Y}$ of 
connected components (this is meant to be a topological
decomposition, i.e. we cut $S$ open along disjoint representatives
of the curves of small extremal length). 
A component $Y\in{\cal Y}$ is called
an \emph{$\epsilon$-thick component} for $x$, and 
$({\cal A}_\epsilon,{\cal Y})$ is the 
\emph{$\epsilon$-thin-thick 
decomposition} of $(S,x)$. Note that for $\chi<\epsilon$,
a $\chi$-thick component is a union of $\epsilon$-thick components.

If $q$ is a quadratic differential
with underlying conformal structure $x$ then we also 
call $({\cal A}_\epsilon,{\cal Y})$ the 
$\epsilon$-thin-thick decomposition 
for $q$, and we call a component $Y$ of ${\cal Y}$ 
an $\epsilon$-thick component for $q$. More generally, we call a
non-peripheral incompressible open connected
subsurface $Y$ of $S$ \emph{$\epsilon$-semi-thick} for $q$ if 
the extremal length of each boundary circle of $Y$ is at most $\epsilon$.
Then the result of cutting $Y$ along all the curves of extremal
length (as curves in $S$) at most $\epsilon$ is a union of $\epsilon$-thick
components.

Let again $q$ be a quadratic differential on $S$.
For $\epsilon\leq \epsilon_0$ and 
an $\epsilon$-semi-thick subsurface $Y\subset S$ for 
$q$ let ${\rm \bf Y}$ be the representative
of $Y$ with $q$-geodesic boundary which 
is disjoint from the interiors of the (possibly degenerate)
flat cylinders foliated
by simple closed $q$-geodesics homotopic to the boundary 
components of $Y$. We call ${\rm \bf Y}$ the 
\emph{geometric representative} of $Y$.
Following Rafi \cite{R07}, if $Y$ is not
a pair of pants then 
we define
${\rm size}_q(Y)$ to be the 
shortest $q$-length of an essential
simple closed curve in 
${\rm \bf Y}$.  If $Y$ is a pair of pants then we define
${\rm size}_q(Y)$ to be the diameter of $Y$.

An \emph{expanding annulus} for a simple closed curve 
$\alpha\subset (S,q)$ is an embedded
annulus $A\subset S$ homotopic to $\alpha$
with the following property. One boundary component
of $A$ is a $q$-geodesic $\psi$, and the second boundary
component is a curve which is equidistant to $\psi$. Moreover,
$A$ does not intersect the interior of a flat cylinder 
foliated by closed $q$-geodesics homotopic to $\alpha$
(see \cite{M92} for details).
 
By Theorem 3.1 of \cite{R10}, 
up to making $\epsilon_0$ smaller we 
may assume that for every quadratic differential $q$ and every
$\alpha\in {\cal A}_{\epsilon_0}={\cal A}_{\epsilon_0}(q)$ 
(this means that we calculate extremal length for the 
conformal structure underlying $q$)
the following holds true.

Assume that $Y,Z$ are the 
$\epsilon_0$-thick components of 
$(S,q)$ which contain $\alpha$ in their boundary.
Note that $Y,Z$ are not necessarily distinct.
Let $E,G$ be the \emph{maximal expanding 
annuli} in the geometric representatives ${\rm \bf Y,Z}$ of 
$Y,Z$ homotopic to $\alpha$. This
means that $E,G$ are expanding annuli contained in 
${\rm \bf Y,Z}$ which are as big as possible, i.e. which
are not proper subsets of another expanding annulus.
If $Y=Z$ then these annuli are required to lie
on the two different sides of $\alpha$ with respect
to an orientation of $\alpha$ and the orientation of $S$. 
Then
\begin{align}\label{extremallength}
\frac{1}{{\rm Ext}_q(\alpha)} &\asymp
\log \frac{{\rm size}_q(Y)}{\ell_q(\alpha)}+
\frac{{\rm size}_q(F_q(\alpha))}{\ell_q(\alpha)}+
\log\frac{{\rm size}_q(Z)}{\ell_q(\alpha)}\\
&\asymp
{\rm Mod}_q(E)+{\rm Mod}_q(F_q(\alpha))+
{\rm Mod}_q(G)\notag\end{align}
where $F_q(\alpha)$ is the (possibly degenerate) flat cylinder
foliated by simple closed geodesics homotopic to $\alpha$ and
where ${\rm Mod}_q$ is the modulus with
respect to the conformal structure defined by $q$. Also, 
$\asymp$ means 
equality up to a uniform multiplicative constant.

Let $M_0>0$ be sufficiently large so that Theorem 5.3 and 
Corollary 5.4 of \cite{R10} hold true for this $M_0$. 
For $M\geq M_0$ and 
for $\epsilon\leq \epsilon_0$, 
a boundary curve $\alpha$ of an 
$\epsilon$-semi-thick 
subsurface $Y$ of $S$ is called \emph{$M$-large}
if  ${\rm Mod}(E)\geq M$ where as before, $E$ is a maximal
expanding annulus homotopic to $\alpha$
in the geometric representative
${\rm \bf Y}$ of $Y$.
An $\epsilon$-semi-thick subsurface $Y$ is \emph{$M$-large}
if $Y$ is not a pair of pants and if 
each of its boundary circles is $M$-large.
We also require that $Y$ is non-trivial, i.e.
that $Y\not=S$.

Since $\epsilon_0<\delta^2/2$ by assumption
and since 
the $q$-length of a simple closed curve $\alpha$ is at most
$\sqrt{{\rm Ext}_q(\alpha)}$, 
a $\delta$-wide simple closed curve $\alpha$
can not have an essential intersection with
a boundary curve of an $\epsilon_0$-thick component
of $S$. As a consequence,
a geodesic representative of $\alpha$ either
is contained in the geometric representative 
${\bf Y}$ of 
an $\epsilon_0$-thick component $Y$ 
of $S$, or it is homotopic to a boundary circle
of such a component.

In the case that $\alpha$ is contained in 
the geometric representative ${\bf Y}$ of an
$\epsilon_0$-thick component $Y$ of $S$ and 
is not homotopic to a
boundary component of $Y$, the embedded annulus of 
width $\delta$ which is 
homotopic to $\alpha$ can intersect
a boundary component of $Y$ only in a set of small
diameter. As a consequence, up to 
adjusting the number $\epsilon_0$, the $q$-diameter
of $Y$ is bounded from below by a universal constant.

If ${\rm Ext}_q(\alpha)\leq \epsilon_0$ then there are two
(not mutually exclusive) possibilities. In the first  case, 
$\alpha$ is the core curve of a flat cylinder of small 
circumference whose width is uniformly bounded from 
below. Such a cylinder has a large modulus. 
The second possibility is that there is an expanding
annulus homotopic to $\alpha$ whose width is uniformly
bounded from below. Then this annulus is contained
in the geometric representative ${\bf Y}$ of an
$\epsilon_0$-thick component $Y$ of $S$. 
The $q$-diameter of
${\bf Y}$ is bounded from below
by a universal constant. 

This observation is used to show

\begin{lemma}\label{wide}
For every $M>0,\chi>0$ there is a number 
$\epsilon_1(\chi,M)<\epsilon_0$ 
with the 
following property. Let $q\in {\cal Q}^1(S)$ and assume
that there
is a simple closed curve $\beta$
with ${\rm Ext}_{q}(\beta)<\epsilon_1(\chi,M)$.
Let $\alpha$ be a 
$\delta$-wide curve for $q$; then 
up to isotopy, either $\alpha$
is contained
in an $M$-large subsurface of $S-\beta$ or
$\alpha$ is contained in a flat cylinder of 
modulus at least $\chi$.
\end{lemma}
\begin{proof} For $\chi>0$ let 
$\nu\leq \epsilon_0$ 
be sufficiently small that the modulus of a flat
cylinder of circumference $\sqrt{\nu}$ and 
width $\delta/3$ is at least $\chi$.

Let $\alpha$ be a $\delta$-wide curve
for the quadratic differential $q$. By the choice 
of $\epsilon_0$, there is a $q$-geodesic representative
of $\alpha$ which is contained in  
a $\nu$-thick component $Y$ for $q$. 

If $\alpha$ is homotopic to a boundary circle of $Y$ 
and if moreover $\alpha$ is the core curve of a 
flat cylinder of width at least $\delta/3$ 
then $\alpha$ is the core curve of a flat
cylinder of modulus at least $\chi$ and 
we are done.

For the remainder of this proof suppose that
no such cylinder exists.
The discussion preceding this lemma shows that
in this case we may assume 
without loss of generality that 
the $q$-diameter of $Y$ is bounded from below by
a universal constant
$\sigma_0=\sigma_0(\chi)>0$.

By the main result of \cite{R07}, since 
there are no essential simple closed curves in $Y$
of extremal length smaller than $\nu$,
${\rm size}_q(Y)$
is comparable to the $q$-diameter of $Y$
(with a comparison factor depending on $\nu$).
Now the $q$-diameter of $Y$ is at least $\sigma_0$
and therefore ${\rm size}_q(Y)$ is bounded from 
below by a number $\sigma_1\leq \sigma_0$ only
depending on $\nu$ and hence only depending on $\chi$.

Let $M>0$ be arbitrary. 
If each boundary circle of $Y$ is 
$M$-large then we are done.
Otherwise let $\beta$ be a boundary component 
of $Y$ which is not $M$-large. 
Since ${\rm size}_q(Y)\geq \sigma_1$, the
estimate (\ref{extremallength}) shows that 
$\ell_q(\beta)\geq \sigma_2$ for a number $\sigma_2>0$ only
depending on $\sigma_1$. In particular, we have
${\rm Ext}_q(\beta)\geq \sigma_2^2$. 

If $\beta$ is non-separating and defines two distinct 
free homotopy classes in $Y$
then $Y\cup \beta=Y_1$ is a subsurface of $S$ containing 
$\alpha$. Since the length of $\beta$ is bounded from
below by a universal constant, the size of $Y\cup Y_1$ 
is bounded from below by a universal constant $\sigma_3>0$
as well. 

Otherwise there is a
$\chi$-thick component $Y^\prime$ of $S$
which contains $\beta$ in its boundary and which is distinct
from $Y$. Using again the fact that the length of 
$\beta$ is uniformly bounded from below, 
the diameter of $Y^\prime$ and hence
the size of $Y^\prime$ is uniformly bounded from below. 
Let  $Y_1=Y\cup Y^\prime$ and note that
$Y_1$ contains $\alpha$, and its size is bounded from 
below by a universal constant $\sigma_4>0$.
In particular, any very short boundary curve
of $Y_1$ is contained in an expanding cylinder of 
modulus at least $M$.

Repeat this reasoning with $Y_1$. In at most $3g-4+m$ 
steps we either  
find an $M$-large subsurface $Y$ of $S$ 
containing $\alpha$ whose size is bounded from below
by a universal constant, 
or we conclude that the shortest $q$-length
of an essential simple closed
curve on $S$ is bounded from below by a universal constant.
Together this shows the lemma.
\end{proof}

Let $X\subset S$
be a non-peripheral, incompressible, 
open connected subsurface which is distinct from 
$S$, a three-holed sphere and
an annulus. 
The \emph{arc and curve complex} 
${\cal C}^\prime(X)$ of $X$ is defined to be the complex
whose vertices are isotopy classes of 
arcs with endpoints on $\partial X$ or essential 
simple closed curves in $X$.
Two such vertices are connected by an edge of length
one if they can be realized disjointly. 
There is a \emph{subsurface projection}
$\pi_X$ of ${\cal C\cal G}$ into 
the space of subsets of  
${\cal C}^\prime(X)$ which associates to 
a simple closed curve the homotopy classes of its
intersection components with $X$ (see \cite{MM00}).
For every simple closed curve $c$, the diameter of $\pi_X(c)$ 
in ${\cal C}^\prime(X)$ is at most one.
Moreover, if $c$ can be realized disjointly from $X$ then 
$\pi_X(c)=\emptyset$.

There also is an arc complex ${\cal C}^\prime(A)$ for 
an essential annulus $A\subset S$,
and there is a 
subsurface projection $\pi_A$ of 
${\cal C\cal G}$ into the space of subsets of 
${\cal C}^\prime(A)$.
We refer to \cite{MM00} for details of this construction.
In the sequel 
we call a subsurface $X$ of $S$  \emph{proper}
if $X$ is non-peripheral, incompressible, open and connected
and different from a three-holed sphere or $S$.

As before, 
let $\delta>0$ be such that for every
$q\in {\cal Q}^1(S)$ there is a $\delta$-wide
curve for $q$. 
The next lemma is a version of Proposition
\ref{progression} for Teichm\"uller geodesic arcs
which are allowed to be entirely contained in the
thin part of Teichm\"uller space.

\begin{lemma}\label{largesize}
For every $\chi>0,k>0$ there 
is a number $T=T(\chi,k)>0$  
with the
following property. Let $q_t$ be 
the cotangent line of a Teichm\"uller geodesic
defined by two simple closed curves
$\xi,\zeta$. For $t\in \mathbb{R}$ let 
$\alpha(t)\in {\cal C\cal G}$ be a $\delta$-wide
curve for $q_t$. Then one of the following
(not mutually exclusive) possibilities is satisfied.
\begin{enumerate}
\item $d_{\cal C\cal G}(\Upsilon(q_0),\Upsilon(q_T))\geq k$.
\item There is a number $t\in[0,T]$ and   
a proper subsurface $Y$ of $S$ containing $\alpha(t)$ 
such that 
${\rm diam}(\pi_Y(\xi\cup \zeta))\geq k$.
\item There is a number $t\in [0,T]$ such that $\alpha(t)$
is the core curve of a flat cylinder of modulus at least 
$\chi$.
\end{enumerate}
\end{lemma}
\begin{proof} Define a \emph{shortest marking} $\mu$
for a quadratic differential $q\in {\cal Q}^1(S)$ to
consist of a  
pants decomposition
with pants curves of the shortest extremal length
for the conformal structure underlying $q$ 
constructed using the greedy algorithm. There is a
system of spanning arcs, one for each pants curve
$\alpha$ of $\mu$. 
Such an arc is a geodesic arc in the annular cover
of $S$ with fundamental group generated by $\alpha$ 
which 
intersects a $q$-geodesic representing $\alpha$ perpendicularly.  
We refer to Section 5
of \cite{R10} for details- the purpose of 
using such shortest markings here is for ease of reference 
to the results of \cite{R10}. 
A shortest marking defines a subset
of ${\cal C\cal G}$. 

In the remainder of this proof
we denote the diameter of a
subset $X$ of ${\cal C\cal G}$ by ${\rm diam}_{\cal C\cal G}(X)$.
Moreover, 
whenever we use constants
for a given surface $S$ of complexity 
$n$, we assume
that they are also valid in the same context
for surfaces of complexity smaller than $n$.

For $q\in {\cal Q}^1(S)$ and $\rho>0$ 
call a curve $\alpha$ 
\emph{$\rho$-slim} 
if $\alpha$ is \emph{not} the core curve of 
a flat cylinder for $q$ of modulus at least $\rho$.
Let $\rho>0,k>0$, let $b\in (0,1)$ and 
let $\lambda$ be the
Lebesgue measure on the real line.
Let $(q_t)$ be the cotangent line of 
a Teichm\"uller geodesic on $S$.
For each $t$ let $\mu_t$ be a
shortest marking for $q_t$.
Let $n=\xi(S)\geq 1$ be the complexity of $S$.
Let $\delta >0$ be as before. We show by induction on $n$
the following

{\bf Claim:} There is 
a number $T_n=T_n(\delta,\rho,k,b)>0$ 
and a number $f(b,n)\in (0,b]$ with the 
following property.
Assume that there is a set $A\subset [0,T_n]$ with
$\lambda(A)\geq bT_n$ such that 
for every $t\in A$ there is a  
$\delta$-wide $\rho$-slim curve $\alpha(t)$
for $q_t$. 
Then either 
\[{\rm diam}_{\cal C\cal G}(\mu_0\cup \mu_{T_n})\geq k\] 
or there 
is a set $A^\prime\subset A$ 
with $\lambda(A^\prime)\geq f(b,n)T_n$
so that for all $t\in A^\prime$, the curve $\alpha(t)$
is contained in a proper
subsurface $Y$ of $S$ 
with 
${\rm diam}(\pi_Y(\mu_{0}\cup 
\mu_{T_n}))\geq k$. 

The claim easily yields the lemma. 
Namely, by Theorem 5.3 of \cite{R10}, 
for every $k>0$ there is a number $\ell=\ell(k)>k$ with the
following property. Let $(q_s)$ be the cotangent line
of a Teichm\"uller geodesic defined
by simple closed curves $\xi,\zeta$. Let $s<t$, let
$\mu_s,\mu_t$ be shortest markings for 
$q_s,q_t$ and assume that there is a subsurface $Y$ of 
$S$ such that ${\rm diam}(\pi_Y(\mu_s\cup \mu_t))\geq \ell$;
then
\[{\rm diam}(\pi_Y(\xi\cup \zeta))\geq k.\]

Let 
$T=T(\chi,k)=T_n(\delta,\chi,\ell,\frac{1}{2})$ and let
$q_t$ be the cotangent line of a Teichm\"uller geodesic
defined by two simple closed curves $\xi,\zeta$.
For $t\in [0,T]$ let $\alpha(t)$ be a $\delta$-wide curve for
$q_t$. If there is some $t\in [0,T]$ such that
$\alpha(t)$ is \emph{not} $\chi$-slim 
then the third property in the statement
of the lemma is satisfied and we are done.
Otherwise each of the curves $\alpha(t)$ $(t\in [0,T])$ 
is $\chi$-slim and hence
we can apply the above claim to obtain
the statement of the lemma.

For the inductive proof of the claim, note that
if $\xi(S)=1$ then $S$ either is a once punctured torus
or a four times punctured sphere. 
Consider first the case of a one-punctured torus.
Quadratic differentials on such a torus are just squares of 
abelian differentials on the torus with the puncture
filled in. As a consequence, a 
$\delta$-wide $\rho$-slim simple closed 
curve in $S$ is the core curve of a flat cyclinder 
(i.e. a cylinder foliated by simple closed
geodesics) whose
modulus is at most $\rho$. 
Such flat structures on tori project to a compact 
subset of the moduli space of tori. 

As a consequence, there is a
number $\epsilon>0$ (depending on $\delta,\rho$) with 
the following property. If there is a $\delta$-wide
$\rho$-slim curve for a quadratic differential $q$ on $S$
then $\nu(q)>\epsilon$. 
The above claim now follows from Proposition \ref{progression}
with $f(b,1)=b$ (here only the first alternative 
in the claim occurs).
The case of the four punctured sphere is completely
analogous and will be omitted.

Now assume that the claim holds true for 
all surfaces of complexity at most $n-1\geq 1$.
Let $S$ be a surface of complexity $n$. Theorem 4.2 and 
Theorem 5.3 of \cite{R10} and their proofs 
(see also \cite{M92}) show the
following. 

Let $(q_t)$ be the cotangent line of a Teichm\"uller geodesic.
Assume that there is an interval
$[a,b]\subset \mathbb{R}$ and 
a proper $\epsilon_0$-semi-thick 
subsurface $Y$ of $S$ which is
$M_0$-large for every $t\in [a,b]$. Here as before,
$M_0>0$ is a constant only depending on the complexity
of $S$, chosen in such a way that the results
of \cite{R10} hold true for this number.
Let ${\rm\bf Y}_t$ be the geometric representative
of $Y$ for the quadratic differential $q_t$ and define
$q_{t,Y}$ to be the quadratic differential obtained
by capping off the boundary components 
with (perhaps degenerate) discs or 
punctured discs (see the proof of 
Theorem 4.2 in \cite{R10} for details).
Then $(q_{t,Y})$ $(t\in [a,b])$ is the cotangent line
of a Teichm\"uller geodesic on $Y$ (with the 
boundary circles of $Y$ closed by discs or 
replaced by punctures). 
Moreover, there is a number $\delta^\prime<\delta$
and a number $\rho^\prime<\rho$ such that
every essential simple closed curve in 
${\rm \bf Y}_t$ which is 
$\delta$-wide and $\rho$-slim for $q_t$ is 
$\delta^\prime$-wide and $\rho^\prime$-slim for $q_{t,Y}$.

Let $T_{n-1}=T_{n-1}(\delta^\prime,\rho^\prime,k^\prime,b/4n)$ 
be the number found for surfaces
of complexity at most $n-1$ and for the
numbers $\delta^\prime<\delta,\rho^\prime<\rho$
and a number $k^\prime>k$ which will be determined
below.
Let $M_1>M_0$ be sufficiently large that the following
holds true. Let $s\in \mathbb{R}$ and let
$Y$ be an $\epsilon_0$-semi-thick subsurface 
for $q_s$ with geometric representative ${\rm \bf Y}_s$.
Let $\beta$ be a boundary circle of $Y$ and let 
$E\subset {\rm \bf Y}_s$ be a maximal  
expanding annulus for $q_s$ 
homotopic to $\beta$. If ${\rm Mod}_{q_s}(E)\geq M_1$ 
then for $\vert t-s\vert\leq T_{n-1}$ the modulus
of the maximal expanding annulus for $q_t$
homotopic to $\beta$ which is contained
in the geometric representative ${\rm \bf Y}_t$ of $Y$
for $q_t$ and lies on the same
side of $\beta$ as $E$ is not smaller than $M_0$.

For this number $M_1$ let $\epsilon_1(\rho/2,M_1)>0$
be as in Lemma \ref{wide}.
Let $\epsilon_2<\epsilon_1(\rho/2,M_1)$ be 
sufficiently small that whenever 
$\beta$ is a curve with ${\rm Ext}_{q_s}(\beta)\leq \epsilon_2$
then ${\rm Ext}_{q_t}(\beta)\leq \epsilon_1(\rho/2,M_1)$
for all $t$ with $\vert s-t\vert \leq T_{n-1}$.

Let $\ell>0$ be sufficiently large that 
for every $q\in {\cal Q}^1(S)$ the 
diameter in ${\cal C\cal G}$ of the union of 
the set of curves
from a shortest marking for $q$ with the set of 
all $\delta$-wide curves does not exceed $\ell$ 
(such a number exists since
the extremal length and hence the hyperbolic
length of a $\delta$-wide
curve is uniformly bounded from  above, see \cite{H10}).
Let $T_n=R(\epsilon_2,k+2\ell,b/8)$ 
be as in Proposition \ref{progression}. We may assume that
\[bT_n/16\geq bT_n/8n\geq T_{n-1}.\]

Let $(q_t)$ be the cotangent line of a Teichm\"uller
geodesic on $S$. 
Proposition \ref{progression} shows that 
either $d_{\cal C\cal G}(\Upsilon(q_0),\Upsilon(q_{T_n}))\geq k+2\ell$
or there is a subset $B$ of $[0,T_n]$ with 
$\lambda(B)\geq (1-b/8)T_n$ and for every 
$s\in B$ there is a curve of extremal length 
at most $\epsilon_2$ for $q_s$. 
In the first case we
conclude from the choice of $\ell$ that
${\rm diam}_{\cal C\cal G}(\mu_0\cup \mu_{q_{T_n}})
\geq k$ and we are done. Thus assume that the second
possibility holds true. 
 
Let us now 
suppose that there exists a set 
$A\subset [0,T_n]$ 
with $\lambda(A)\geq bT_n$ such that 
for every $t\in A$ there is a $\delta$-wide 
$\rho$-slim curve 
$\alpha(t)$ for $q_t$.
Let $\chi_B,\chi_A$ be the characteristic function of $B,A$,
respectively. Since $bT_n/16\geq T_{n-1}$ and 
$\lambda(B)\geq (1-b/8)T_n$, 
for every $s\in [0,T_{n-1}]$ we have
\[\lambda\{u\in [T_{n-1},T_n-T_{n-1}]\mid u-s\in B\}
\geq (1-b/4)T_n.\]
Now $\lambda(A)\cap [T_{n-1},T_n-T_{n-1}]\geq 3bT_n/4$
and hence  
Fubini's theorem shows that 
\begin{align}
\int_0^{T_n-T_{n-1}} \bigl( \frac{1}{T_{n-1}}
\int_s^{s+T_{n-1}}\chi_B(s)\chi_A(u) du \bigr) ds &\notag\\
\geq  \frac{1}{T_{n-1}}\int_0^{T_{n-1}}\int_{T_{n-1}}^{T_n-T_{n-1}}
\chi_B(u-s)\chi_A(u)duds  &\geq 
 bT_n/2.\notag
 \end{align}
As a consequence, the Lebesgue measure of the set 
\[C=\{s\in B\cap [0,T_n-T_{n-1}]
\mid \int_s^{s+T_{n-1}}\chi_A(u)du\geq bT_{n-1}/4\}\] 
is at least $bT_{n}/4$.

If $s\in C$ then $s\in B\cap [0,T_n-T_{n-1}]$ 
and hence there is a simple closed
curve $\beta_s$ with ${\rm Ext}_{q_s}(\beta_s)\leq \epsilon_2$.
Moreover, there is a set 
\[D\subset [s,s+T_{n-1}]\] 
of Lebesgue
measure at least $bT_{n-1}/4$ so that for every 
$t\in D$ there is a   
$\delta$-wide $\rho$-long curve $\alpha(t)$
for $q_t$. 

By Lemma \ref{wide} and the choice of $\epsilon_2$,
for every $t\in D$ the curve $\alpha(t)$   
is contained in an $M_1$-large
subsurface $Y\subset S-\beta_s$ for $q_t$. 

A subsurface $Y$ of $S$ which is 
$M_1$-large for some $t\in [s,s+T_{n-1}]$ is 
$M_0$-large for every $t^\prime\in [s,s+T_{n-1}]$.
Now $M_0$-large subsurfaces for $q_t$ 
are pairwise disjoint, and their
number is at most $n$. As a consequence, 
there is a subsurface $Y\subset S-\beta_s$ which is $M_0$-large
for every $t\in [s,s+T_{n-1}]$, and
there is a subset $D_s\subset D$ of Lebesgue measure at least
$T_{n-1}b/4n$ so that 
$\alpha(t)\in Y$ for every $t\in D_s$.

Let $\psi(s),\psi(s+T_{n-1})\subset Y$ be 
shortest markings for 
$q_{s,Y},q_{s+T_{n-1},Y}$.
By the choice of $T_{n-1}$, we can apply the 
induction hypothesis to the cotangent line
$(q_{t,Y})$ $(t\in [s,s+T_{n-1}])$ of the induced
Teichm\"uller geodesic on $Y$. 

Let $H_s\subset [s,s+T_{n-1}]$
be the set of all $t\in D_s$ such that the curve $\alpha(t)$  
is contained in a proper subsurface $Z_t$ of $Y$
with the additional property that 
\[{\rm diam}(\pi_{Z_t}(\psi(s)\cup\psi(s+T_{n-1})))\geq k^\prime.\]
By induction assumption, either the diameter 
of $\psi(s)\cup \psi(s+T_{n-1})$ in 
the curve graph of $Y$ is at least $k^\prime$ 
or 
\[\lambda(H_s)\geq f(b/4n,n-1)T_{n-1}.\]

Theorem 5.3 of \cite{R10} shows that
there is a number $p>0$ with the following property.
If ${\rm diam}(\pi_Y(\psi(s)\cup \psi(s+T_{n-1})))\geq k^\prime$ 
then 
${\rm diam}(\pi_Y(\mu_0\cup 
\mu_{T_n}))\geq k^\prime-p$.
Moreover, for $t\in H_s$ 
the diameter of the subsurface projection of 
$\mu_0\cup \mu_{T_n}$ 
into $Z_t$ is at least $k^\prime-p$. 

Define $k^\prime=k+p$ and let
$H_s^\prime=D_s$ if 
${\rm diam}(\pi_Y(\psi(s)\cup \psi(s+T_{n-1})))\geq k^\prime$ 
(which implies that 
${\rm diam}(\pi_Y(\mu_0\cup \mu_{T_n}))\geq k$), and 
define $H_s^\prime=H_s$ otherwise.
Then 
\[\lambda(H_s^\prime)\geq f(b/4n,n-1)T_{n-1}\] 
for every $s\in C$, moreover 
for all $t\in H_s^\prime$ the curve
$\alpha(t)$ is contained in a subsurface $Y_s$ of $S$ 
with 
${\rm diam}(\pi_{Y_s}(\mu_0\cup \mu_{T_n}))\geq k$. 
 
For $s\in C$ 
let $\chi_s$ be the characteristic function of 
$D_s$, viewed as a subset of $\mathbb{R}$, and let 
$\chi(t)=\max\{\chi_s(t)\mid s\}$. 
Since 
$\lambda(C)\geq bT_n/4$ we have 
\[\int_0^{T_n}\chi(s)ds\geq f(b/4n,n-1)bT_n/4\]
and hence the claim holds true with $f(b,n)=bf(b/4n,n-1)/4$.
\end{proof}

Now we are ready for the proof of the Theorem from the 
introduction. To this end choose 
a number
$K_0>0$ such that the following holds true.
Let $\xi,\zeta$ be simple closed curves. If there is a 
proper subsurface $Y$ of $S$ such that
${\rm diam}(\pi_Y(\xi\cup \zeta))\geq K_0$ then a 
geodesic $\gamma$ in ${\cal C\cal G}$ 
connecting $\xi$ to $\zeta$ passes through
the complement of $Y$. The existence of
such a number $K_0>0$ follows from 
Theorem 3.1 of \cite{MM00}.

Let $(q_t)$ be the cotangent line of 
the Teichm\"uller geodesic connecting $\xi$ to 
$\zeta$. 
By the results of \cite{MM00}, there is
a number $\chi>0$ with the following property.
If $\alpha\not=\xi,\zeta$ is a simple closed curve so that 
$\alpha$ is the core curve of a flat cylinder
for $q_t$ of modulus at least $\chi$ then 
the diameter of 
the subsurface projection of $(\xi,\zeta)$ into
an annulus with core curve $\alpha$ 
is at least $K_0$.

Let $\kappa>0$ be as in (\ref{kappa}).
For these numbers $\chi,\kappa,K_0$ let 
$R=T(\chi,K_0+4\kappa)$ be as in Lemma \ref{largesize}.

Let $t_0<\dots <t_n=T$ be as in the
statement of the theorem.
Let $u$ be an integer bigger than $R$.
By Lemma \ref{largesize} and the above discussion, for all 
$\ell$ we either have 
\begin{enumerate}\item[i)]
$d_{\cal C\cal G}(\Upsilon(q_{t_\ell}),
\Upsilon(q_{t_{\ell+u}}))\geq K_0+4\kappa$ 
and hence 
$d_{\cal C\cal G}(\Upsilon(q_{t_v}),
\Upsilon(q_{t_w}))
\geq 2\kappa$ 
for all $w\geq \ell+u\geq \ell \geq v$, or
\item[ii)]  
there is a subsurface $Y$ of $S$ 
(perhaps an annulus) so that the diameter of 
the subsurface projection of $\xi,\zeta$ into $Y$ is 
at least $K_0$ and that $\Upsilon(q_{s})
\subset Y$ for some $s\in [\ell,\ell+u]$.
\end{enumerate}

It now suffices to show that 
$d_{\cal C\cal G}(\Upsilon(q_\ell),\Upsilon(q_{t_{\ell +u(8\kappa+8)}}))
\geq 2\kappa$ for all $\ell$.
By the choice of $\kappa$, this holds true if 
there is some $j\leq u(8\kappa+6)$ so that the alternative
i) above is satisfied for $q_{t_j}$.

Otherwise for every $i\leq 8\kappa+6$ there is some $s\in [ui,u(i+1)]$ 
such that
$\Upsilon(q_s)$ is contained in a subsurface $Y$  
with the property that the diameter of the 
subsurface projection of $\xi,\zeta$ into $Y$ is at least $K_0$.
By the choice of $K_0$ (see Theorem 3.1 of \cite{MM00}), 
there is a curve $c_s$ on $\gamma$ disjoint from $\Upsilon(q_{s})$.
In particular, we have
$d_{\cal C\cal G}(c_s,\Upsilon(q_{s}))\leq 1$.

Now if $v\in [uj,u(j+1)]$ for some $j$ with
$\vert i-j\vert \geq 2$
then $d_{\cal C\cal G}(c_s,c_v)\geq 1$ by assumption.
In particular, $c_s,c_v$ are distinct.
As a consequence, there are at least $4\kappa+3$ distinct points on the
geodesic $\gamma$ arising in this way.
Then there are at least two of these points,
say the points $c_s,c_v$, whose distance
is at least
$4\kappa +2$. This shows 
$d_{\cal C\cal G}(\Upsilon(q_{t_s}),
\Upsilon(q_{t_v}))\geq 4\kappa$ which implies the
required estimate.

\bigskip\bigskip

\noindent
MATH. INSTITUT DER UNIVERSIT\"AT BONN\\
ENDENICHER ALLEE 60\\ 
53115 BONN\\ 
GERMANY\\

\bigskip\noindent
e-mail: ursula@math.uni-bonn.de


\begin{thebibliography}{MMS10}



\bibitem[B06]{B06} B. Bowditch, {\em Intersection numbers and the
hyperbolicity of the curve complex}, 
J. Reine Angew. Math. 598 (2006), 105--129.


\bibitem[G87]{G87} F.~Gardiner,
{\em Teichm\"uller theory and quadratic differentials},
Pure and Appl. Math., Wiley-Interscience,
New York 1987.


\bibitem[H06]{H06} U.~Hamenst\"adt, {\em Train tracks
and the Gromov boundary of the complex of curves},
in ``Spaces of Kleinian groups'' (Y.~Minsky, M.~Sakuma,
C.~Series, eds.), London Math. Soc. Lec. Notes 329 (2006),
187--207.







\bibitem[H10]{H10} U.~Hamenst\"adt, {\em Stability
of quasi-geodesics in Teichm\"uller space},
Geom. Dedicata 146 (2010), 101--116.

\bibitem[H12]{H12} U.~Hamenst\"adt,
{\em Teichm\"uller theory}, Lecture Notes of the
Park City Summer School 2011. 









\bibitem[MM99]{MM99} H.~Masur, Y.~Minsky, 
{\em Geometry of the complex of curves I: Hyperbolicity},
Invent. Math. 138 (1999), 103--149.

\bibitem[MM00]{MM00} H. Masur, Y. Minsky, 
{\em Geometry of the complex of curves II: Hierachical structure},
Geom. Funct. Anal. 10 (2000), 902--974.





\bibitem[M92]{M92} Y.~Minsky, {\em Harmonic
maps, length, and energy in Teichm\"uller space},
J. Diff. Geom. 35 (1992), 151--217.





\bibitem[R07]{R07} K.~Rafi, {\em Thick-thin decomposition
for quadratic differentials}, Math. Res. Lett. 14 (2007),
333--341.

\bibitem[R10]{R10} K.~Rafi, {\em Hyperbolicity
in the Teichm\"uller space}, arXiv:1011.6004


\bibitem[Sh04]{S04} K.~Shackleton, {\em Tightness
and computing distances in the curve complex},
arXiv:math/0412078.

\bibitem[S84]{S84} K.~Strebel, {\em Quadratic differentials},
Ergebnisse der Mathematik 5, Springer, Berlin 1984.


\end{thebibliography}
\end{document}